\documentclass[a4paper]{article}

\usepackage[a4paper,left=2.5cm,right=2.5cm,top=2.5cm,bottom=2.5cm]{geometry}

\usepackage{amsmath,amssymb,amsfonts,mathtools}
\usepackage{amsthm}

\usepackage{graphicx}
\usepackage{indentfirst}
\usepackage{array}
\usepackage{mathrsfs}
\usepackage{xcolor}
\usepackage{cite}
\usepackage[colorlinks=true,linkcolor=blue,citecolor=blue,urlcolor=blue]{hyperref}

\numberwithin{equation}{section}

\theoremstyle{plain}
\newtheorem{theorem}{Theorem}[section]
\newtheorem{lemma}[theorem]{Lemma}
\newtheorem{proposition}[theorem]{Proposition}

\theoremstyle{definition}

\theoremstyle{remark}
\newtheorem{remark}[theorem]{Remark}

\newcommand{\N}{\mathbb{N}}

\newcommand{\R}{\mathbb{R}}

\newcommand{\supp}{\operatorname{supp}}
\newcommand{\Vol}{\operatorname{Vol}}
\newcommand{\diam}{\operatorname{diam}}
\newcommand{\inj}{\operatorname{inj}}

\newcommand{\Isom}{\operatorname{Isom}}
\newcommand{\Card}{\operatorname{Card}}

\title{Improved Fractional Sobolev Embeddings on Closed Riemannian Manifolds under Isometric Group Actions}

\author{Hao Tan, Zhipeng Yang\thanks{Corresponding author. E-mail: yangzhipeng326@163.com}}

\date{}

\AtEndDocument{%
  \par
  \bigskip
  \bigskip

  \noindent
  \textbf{Hao Tan}\\[0.2em]
  \textsc{Department of Mathematics, Yunnan Normal University, Kunming, China}\\[0.3em]
  \textit{E-mail address}: \texttt{3332502677@qq.com}\\[1.5em]

  \noindent
  \textbf{Zhipeng Yang}\\[0.2em]
  \textsc{Department of Mathematics, Yunnan Normal University, Kunming, China}\\
  \textsc{Yunnan Key Laboratory of Modern Analytical Mathematics and Applications, Kunming, China}\\[0.3em]
  \textit{E-mail address}: \texttt{yangzhipeng326@163.com}%
}

\begin{document}

\maketitle

\textbf{Abstract.}
In this paper, we study symmetry-improved fractional Sobolev embeddings on closed Riemannian manifolds under the action of compact isometry groups. We prove that \(G\)-invariant fractional Sobolev spaces embed into higher \(L^p\) spaces, with corresponding compactness results depending on the minimal orbit dimension. We also investigate the associated optimal constants in the improved critical inequality and in the standard critical inequality under finite-orbit symmetry.

\vspace{3mm}

\vspace{3mm}
\noindent\textbf{MSC2020 Mathematics Subject Classification:} 46E35, 58J05, 53C21.

% \tableofcontents

\section{Introduction and main results}

Let \((M,g)\) be a closed Riemannian \(n\)-manifold.
For the classical first-order Sobolev space \(H_1^q(M)\), the standard Sobolev embedding theorem asserts that, for \(1\le q<n\),
\[
H_1^q(M)\hookrightarrow L^p(M)
\]
continuously for every \(1\le p\le q^*=\frac{nq}{n-q}\), and compactly for every \(1\le p<q^*\).
In the presence of symmetries, Hebey and Vaugon \cite{MR1481970,MR1489942} showed that the critical dimension can be improved by replacing \(n\) with the dimension of the quotient space.
More precisely, if \(G\subset \Isom_g(M)\) is compact, if every orbit \(O_G^x\) is infinite, and if
\[
k=\min_{x\in M}\dim O_G^x,
\]
then \(k\ge 1\), and the corresponding \(G\)-invariant Sobolev embeddings improve from dimension \(n\) to dimension \(n-k\).
Related results were also obtained in special situations by Lions \cite{MR850686} and Ding \cite{MR863646}.

The purpose of this paper is to establish analogous results for intrinsic fractional Sobolev spaces on closed manifolds and to investigate the corresponding optimal constants under symmetry constraints.

In the Euclidean setting, for \(s\in(0,1)\), \(p\in[1,\infty)\), and \(sp<n\), the fractional Sobolev space \(W^{s,p}(\R^n)\) is defined by
\begin{equation}\label{eq1.1}
W^{s,p}(\R^n)
=
\bigl\{
u\in L^p(\R^n): [u]_{s,p}<\infty
\bigr\},
\end{equation}
where
\begin{equation}\label{eq1.2}
[u]_{s,p}
=
\left(
\iint_{\R^n\times\R^n}
\frac{|u(x)-u(y)|^p}{|x-y|^{n+sp}}\,dx\,dy
\right)^{1/p}.
\end{equation}
On a closed Riemannian manifold \((M,g)\), Guo, Peng, and Xi \cite{MR3856432} introduced the geodesic-distance version of this space, and it was later shown in \cite{MR3975112} that a Brezis--Bourgain--Mironescu type formula still holds.

In this paper we use the intrinsic heat-kernel definition developed in \cite{Tan2025SharpFS}.
Let \(K_M(t,x,y)\) be the heat kernel of \((M,g)\).
For \(s\in(0,1)\) and \(p\in[1,\infty)\), define
\[
K_p^s(x,y)
=
\frac{1}{\left|\Gamma\left(-\frac{sp}{2}\right)\right|}
\int_0^\infty
K_M(t,x,y)\,
\frac{dt}{t^{1+\frac{sp}{2}}},
\qquad x\ne y,
\]
and, for \(u\in L^p(M)\),
\begin{equation}\label{eq1.3}
[u]_{W^{s,p}(M)}^p
=
\iint_{M\times M}
|u(x)-u(y)|^p K_p^s(x,y)\,d\mu(x)\,d\mu(y).
\end{equation}
The intrinsic fractional Sobolev space is then defined by
\[
W^{s,p}(M)
=
\bigl\{
u\in L^p(M): [u]_{W^{s,p}(M)}<\infty
\bigr\},
\]
and is endowed with the norm
\[
\|u\|_{W^{s,p}(M)}
=
\|u\|_{L^p(M)}+[u]_{W^{s,p}(M)}.
\]
Moreover, \([u]_{W^{s,p}(M)}\) is equivalent to the geodesic-distance Gagliardo seminorm
\[
\left(
\iint_{M\times M}
\frac{|u(x)-u(y)|^p}{d_g(x,y)^{n+sp}}\,d\mu(x)\,d\mu(y)
\right)^{1/p}.
\]

By \cite{Tan2025SharpFS}, if \(sq<n\) and
\[
q_s^*=\frac{nq}{n-sq},
\]
then
\[
W^{s,q}(M)\hookrightarrow L^p(M)
\]
continuously for every \(1\le p\le q_s^*\), and compactly for every \(1\le p<q_s^*\).

Let \(G\subset \Isom_g(M)\) be a compact subgroup.
We define the \(G\)-invariant subspace by
\[
W_G^{s,q}(M)
=
\bigl\{
u\in W^{s,q}(M): u\circ \sigma=u \text{ for all } \sigma\in G
\bigr\}.
\]
If every orbit is infinite, we set
\[
k=\min_{x\in M}\dim O_G^x.
\]
In this case \(k\ge 1\), and the natural improved critical exponent is
\[
\widetilde q_s^*
=
\frac{(n-k)q}{n-k-sq},
\qquad \text{provided } n-k>sq.
\]
We say that the inequality
\begin{equation}\label{eq1.4}
\|u\|_{L^{\widetilde q_s^*}(M)}
\le
A [u]_{W^{s,q}(M)}
+
B \|u\|_{L^q(M)}
\qquad \text{for all } u\in W_G^{s,q}(M)
\end{equation}
is \(G\)-valid.

When \(n-k>sq\), we define
\[
\mathcal A_{q,G}(M)
=
\bigl\{
A\in \R:\ \exists B\in \R \text{ such that \eqref{eq1.4} is \(G\)-valid}
\bigr\},
\]
\[
\mathcal B_{q,G}(M)
=
\bigl\{
B\in \R:\ \exists A\in \R \text{ such that \eqref{eq1.4} is \(G\)-valid}
\bigr\},
\]
and
\[
\alpha_{q,G}(M)=\inf \mathcal A_{q,G}(M),
\qquad
\beta_{q,G}(M)=\inf \mathcal B_{q,G}(M).
\]
Thus, \(\alpha_{q,G}(M)\) and \(\beta_{q,G}(M)\) are the optimal first and second constants in the improved fractional Sobolev inequality under the \(G\)-invariance constraint.

Our first result establishes the improved \(L^p\)-embedding for \(G\)-invariant functions.

\begin{theorem}\label{thm1.1}
Let \((M,g)\) be a closed Riemannian \(n\)-manifold, let \(s\in(0,1)\), let \(q\in[1,\infty)\), and let \(G\subset \Isom_g(M)\) be compact.
Assume that
\[
\Card O_G^x=+\infty
\qquad \text{for every } x\in M,
\]
and set
\[
k=\min_{x\in M}\dim O_G^x.
\]
Then \(k\ge 1\), and the following assertions hold.
\begin{enumerate}
\item[(i)] If \(n-k<sq\), then
\[
W_G^{s,q}(M)\hookrightarrow L^p(M)
\]
continuously and compactly for every \(p\ge 1\).

\item[(ii)] If \(n-k>sq\), then
\[
W_G^{s,q}(M)\hookrightarrow L^p(M)
\]
continuously for every
\[
1\le p\le \frac{(n-k)q}{n-k-sq},
\]
and compactly for every
\[
1\le p<\frac{(n-k)q}{n-k-sq}.
\]
\end{enumerate}
\end{theorem}

Our second result identifies the optimal second constant in the improved critical inequality.

\begin{theorem}\label{thm1.2}
Under the assumptions of Theorem~\ref{thm1.1}, assume in addition that
\[
n-k>sq.
\]
Then
\[
\beta_{q,G}(M)=\Vol(M)^{-s/(n-k)},
\]
and there exists \(A>0\) such that
\[
\|u\|_{L^{\widetilde q_s^*}(M)}
\le
A [u]_{W^{s,q}(M)}
+
\Vol(M)^{-s/(n-k)}\|u\|_{L^q(M)}
\]
for every \(u\in W_G^{s,q}(M)\).
\end{theorem}

\begin{remark}\label{rem1.3}
If all \(G\)-orbits are finite, then the improved exponent \(\widetilde q_s^*\) is no longer relevant, and one falls back to the standard critical exponent
\[
q_s^*=\frac{nq}{n-sq}.
\]
In this case, following \cite{Tan2025SharpFS}, one obtains
\[
\|u\|_{L^{q_s^*}(M)}
\le
A [u]_{W^{s,q}(M)}
+
\Vol(M)^{-s/n}\|u\|_{L^q(M)}
\]
for all \(u\in W_G^{s,q}(M)\), so the optimal second constant is \(\Vol(M)^{-s/n}\).
\end{remark}

Our third result concerns the optimal first constant in the standard critical inequality under orbit-cardinality constraints.

\begin{theorem}\label{thm1.4}
Let \((M,g)\) be a closed Riemannian \(n\)-manifold, let \(s\in(0,1)\), let \(q\in[1,\infty)\) satisfy \(sq<n\), and let \(G\subset \Isom_g(M)\) be compact.
Define
\[
k=\min_{x\in M}\Card O_G^x\in \N\cup\{+\infty\}.
\]
Then for every \(\varepsilon>0\) there exists \(B_\varepsilon\in \R\) such that
\[
\left(\int_M |u|^{q_s^*}\,d\mu\right)^{q/q_s^*}
\le
\left(
\frac{K(n,s,q)}{k^{sq/n}}+\varepsilon
\right)
\iint_{M\times M}
|u(x)-u(y)|^q K_q^s(x,y)\,d\mu(x)\,d\mu(y)
+
B_\varepsilon\int_M |u|^q\,d\mu
\]
for every \(u\in W_G^{s,q}(M)\).
Here \(K(n,s,q)\) is the sharp Euclidean constant from \cite{Tan2025SharpFS}, and, by convention,
\[
\frac{K(n,s,q)}{k^{sq/n}}=0
\qquad \text{if } k=+\infty.
\]
\end{theorem}

The paper is organized as follows.
In Section~2, we present the preliminary material needed in the sequel, including basic facts from Riemannian geometry, compact group actions, and fractional Sobolev spaces.
Section~3 contains the proof of Theorem~\ref{thm1.1}.
In Section~4, we prove Theorem~\ref{thm1.2} and determine the optimal second constant in the improved critical inequality.
Section~5 is devoted to the proof of Theorem~\ref{thm1.4}, concerning the optimal first constant in the standard critical inequality under symmetry constraints.

\section{Preliminary results}

In this section, we collect some standard facts from Riemannian geometry, compact group actions, and fractional Sobolev spaces that will be used throughout the paper.
For background material, we refer to Chavel \cite{MR1271141}, do Carmo \cite{MR1138207}, Gallot--Hulin--Lafontaine \cite{MR1083149}, Hebey \cite{MR1481970}, and Jost \cite{MR1351009}.

Let \((M,g)\) be a closed Riemannian manifold.
If \(\gamma:[a,b]\to M\) is a piecewise \(C^1\) curve, its length is defined by
\[
L(\gamma)=\int_a^b |\dot\gamma(t)|_g\,dt
=\int_a^b \sqrt{g_{\gamma(t)}(\dot\gamma(t),\dot\gamma(t))}\,dt.
\]
For \(x,y\in M\), let \(C_{x,y}^1\) denote the set of piecewise \(C^1\) curves \(\gamma:[a,b]\to M\) such that \(\gamma(a)=x\) and \(\gamma(b)=y\).
The Riemannian distance associated with \(g\) is defined by
\[
d_g(x,y)=\inf_{\gamma\in C_{x,y}^1}L(\gamma).
\]
This distance induces the original topology of \(M\).

The Riemannian volume measure on \(M\) will be denoted by \(d\mu=d\mu_g\).
In local coordinates, it is given by
\[
d\mu_g=\sqrt{\det(g_{ij})}\,dx^1\cdots dx^n.
\]

Let \(\Isom_g(M)\) denote the isometry group of \((M,g)\).
Since \(M\) is closed, \(\Isom_g(M)\) is a compact Lie group.
Hence every compact subgroup \(G\subset \Isom_g(M)\) is a Lie subgroup.
For such a subgroup \(G\), we define
\[
C_G^\infty(M)
=
\bigl\{u\in C^\infty(M):u\circ \sigma=u \text{ for all } \sigma\in G\bigr\},
\]
and
\[
W_G^{s,q}(M)
=
\bigl\{u\in W^{s,q}(M):u\circ \sigma=u \text{ for all } \sigma\in G\bigr\}.
\]
When needed, we also write \(\mathcal D_G(M)=C_G^\infty(M)\), since \(M\) is compact.

We next recall some facts concerning Riemannian submersions.
Let \((M,g)\) and \((N,h)\) be smooth Riemannian manifolds, and let \(\Pi:M\to N\) be a submersion.
We say that \(\Pi\) is a Riemannian submersion if, for every \(x\in M\), the differential
\[
\Pi_*(x):(H_x,g(x))\to (T_{\Pi(x)}N,h(\Pi(x)))
\]
is an isometry, where \(H_x\) denotes the orthogonal complement of \(T_x(\Pi^{-1}(\Pi(x)))\) in \(T_xM\).

Assume now that \(\dim M>\dim N\), that \(\Pi:M\to N\) is a Riemannian submersion, and that every fiber \(\Pi^{-1}(y)\) is compact.
Define
\[
v(y)=\Vol_g\bigl(\Pi^{-1}(y)\bigr), \qquad y\in N.
\]
Then, for every measurable function \(\phi:N\to \R\) such that \(\phi\,v\in L^1(N,d\mu_h)\), one has
\begin{equation}\label{eq2.1}
\int_M (\phi\circ \Pi)\,d\mu_g
=
\int_N \phi(y)\,v(y)\,d\mu_h(y).
\end{equation}
See, for instance, \cite{MR282313}.

We also recall O'Neill's formula.
If \(\Pi:M\to N\) is a Riemannian submersion, then for any orthonormal vector fields \(X\) and \(Y\) on \(N\), with horizontal lifts \(\widetilde X\) and \(\widetilde Y\), one has
\[
K_{(N,h)}(X,Y)
=
K_{(M,g)}(\widetilde X,\widetilde Y)
+\frac34 \bigl|[\widetilde X,\widetilde Y]^v\bigr|^2,
\]
where \(K_{(M,g)}\) and \(K_{(N,h)}\) denote the sectional curvatures of \((M,g)\) and \((N,h)\), respectively, and \([\widetilde X,\widetilde Y]^v\) denotes the vertical component of the Lie bracket.
See \cite{MR867684,MR909697}.
In particular, if the sectional curvature of \((M,g)\) is bounded from below, then so is the sectional curvature of \((N,h)\), and consequently the Ricci curvature of \((N,h)\) is also bounded from below.

We now turn to compact group actions.
Let \(G\) be a compact subgroup of \(\Isom_g(M)\).
For \(x\in M\), we denote by
\[
O_G^x=\{\sigma(x):\sigma\in G\}
\]
the orbit of \(x\), and by
\[
S_G^x=\{\sigma\in G:\sigma(x)=x\}
\]
the isotropy subgroup at \(x\).
It is classical that \(O_G^x\) is a smooth compact submanifold of \(M\), that the quotient manifold \(G/S_G^x\) exists, and that the canonical map
\[
\Phi_x:G/S_G^x\to O_G^x
\]
is a diffeomorphism.
We refer to \cite{MR717390} for these facts.

An orbit \(O_G^x\) is called principal if, for every \(y\in M\), the isotropy subgroup \(S_G^y\) contains a subgroup conjugate to \(S_G^x\).
Let
\[
\Omega=\bigcup_{\{x:\,O_G^x\text{ is principal}\}} O_G^x.
\]
Then the following properties hold:
\begin{enumerate}
\item \(\Omega\) is an open dense subset of \(M\);
\item if \(\Pi:M\to M/G\) denotes the canonical projection, then \(M/G\) is Hausdorff and \(\Pi\) is proper;
\item \(\Pi(\Omega)=\Omega/G\) has the structure of a smooth connected manifold, and the restriction \(\Pi|_\Omega:\Omega\to \Omega/G\) is a smooth submersion.
\end{enumerate}
See again \cite{MR413144}.

Moreover, the metric \(g\) induces a quotient metric \(h\) on \(\Omega/G\) such that \(\Pi|_\Omega:(\Omega,g)\to (\Omega/G,h)\) is a Riemannian submersion.
The associated distance on the quotient extends to \(M/G\) by
\[
d_h(\Pi(x),\Pi(y))
=
d_g(O_G^x,O_G^y),
\qquad x,y\in M,
\]
where
\[
d_g(O_G^x,O_G^y)=\inf\{d_g(x',y'):x'\in O_G^x,\ y'\in O_G^y\}.
\]
We refer to \cite{MR909697} for the corresponding constructions.

\begin{proposition}[{\cite{Tan2025SharpFS}}]\label{pro2.1}
Let \(s\in(0,1)\) and \(q\in[1,\infty)\).
Then \(W^{s,q}(M)\) is a Banach space.
Moreover, \(C^\infty(M)\) is dense in \(W^{s,q}(M)\).
If \(1<q<\infty\), then \(W^{s,q}(M)\) is reflexive.
\end{proposition}

\begin{lemma}\label{lem2.2}
Let \(G\) be a compact subgroup of \(\Isom_g(M)\).
Then \(C_G^\infty(M)\) is dense in \(W_G^{s,q}(M)\).
\end{lemma}

\begin{proof}
Let \(u\in W_G^{s,q}(M)\).
By Proposition~\ref{pro2.1}, there exists a sequence \((u_j)\subset C^\infty(M)\) such that
\[
u_j\to u
\qquad \text{in } W^{s,q}(M).
\]
Let \(d\sigma\) be the normalized Haar measure on \(G\), and define
\[
v_j(x)=\int_G u_j(\sigma x)\,d\sigma.
\]
Then \(v_j\in C_G^\infty(M)\).

Since every \(\sigma\in G\) is an isometry and the kernel \(K_q^s\) is invariant under isometries, composition with \(\sigma\) preserves the \(W^{s,q}(M)\)-norm.
Hence, by Minkowski's inequality,
\[
\|v_j-u\|_{W^{s,q}(M)}
\le
\int_G \|u_j\circ \sigma-u\circ \sigma\|_{W^{s,q}(M)}\,d\sigma
=
\|u_j-u\|_{W^{s,q}(M)}.
\]
Therefore \(v_j\to u\) in \(W^{s,q}(M)\).
Since \(u\) is \(G\)-invariant, this proves the density of \(C_G^\infty(M)\) in \(W_G^{s,q}(M)\).
\end{proof}

\section{Proof of Theorem \ref{thm1.1}}

We begin with a basic fact concerning orbit dimensions.

\begin{proposition}\label{pro3.1}
Let \(G\) be a compact Lie group acting smoothly on a closed manifold \(M\).
Then every orbit \(O_G^x\) is compact.
Moreover,
\[
\Card O_G^x<\infty \quad \Longleftrightarrow \quad \dim O_G^x=0.
\]
In particular, if \(O_G^x\) is infinite, then \(\dim O_G^x\ge 1\).
\end{proposition}

\begin{proof}
Fix \(x\in M\), and consider the orbit map
\[
\Phi_x:G\to M,\qquad \Phi_x(\sigma)=\sigma(x).
\]
Since \(G\) is compact and \(\Phi_x\) is continuous, the orbit
\[
O_G^x=\Phi_x(G)
\]
is compact.

It is standard that \(O_G^x\) is a smooth immersed submanifold of \(M\), diffeomorphic to \(G/S_G^x\), and hence a compact manifold.
If \(\dim O_G^x=0\), then \(O_G^x\) is a compact \(0\)-dimensional manifold, therefore discrete and compact, and thus finite.
Conversely, if \(O_G^x\) is finite, then it is a discrete submanifold, so \(\dim O_G^x=0\).
The final assertion follows immediately.
\end{proof}

We next recall the standard local normal form for the study of invariant functions.

\begin{lemma}\label{lem3.2}
Let \((M,g)\) be a closed Riemannian \(n\)-manifold, let \(G\) be a compact subgroup of \(\Isom_g(M)\), and let \(x\in M\).
Set
\[
k_x=\dim O_G^x.
\]
Then there exists a coordinate chart \((\Omega,\varphi)\) centered at \(x\) such that
\[
\varphi(\Omega)=U\times V,
\]
where \(U\subset \R^{k_x}\) and \(V\subset \R^{n-k_x}\) are open sets, and such that
\[
U\times \{\Pi_2(\varphi(y))\}\subset \varphi(O_G^y\cap \Omega)
\qquad \text{for every } y\in \Omega,
\]
where \(\Pi_2:\R^{k_x}\times \R^{n-k_x}\to \R^{n-k_x}\) denotes the second projection.
\end{lemma}

\begin{proof}
This is the local description of a compact group action near an orbit.
See Hebey \cite[Lemma 9.1]{MR1688256}.
\end{proof}

Since \(M\) is compact, Lemma \ref{lem3.2} yields a finite covering
\[
M=\bigcup_{m=1}^N \Omega_m
\]
with charts \((\Omega_m,\varphi_m)\) such that, for every \(m\),
\begin{enumerate}
\item[(i)] \(\varphi_m(\Omega_m)=U_m\times V_m\), where \(U_m\subset \R^{k_m}\), \(V_m\subset \R^{d_m}\), and
\[
d_m=n-k_m,\qquad k_m\ge k;
\]
\item[(ii)] \(U_m\) and \(V_m\) are bounded, \(V_m\) has smooth boundary, and \(U_m\) contains a Euclidean ball;
\item[(iii)] for every \(y\in \Omega_m\),
\[
U_m\times \{\Pi_2(\varphi_m(y))\}\subset \varphi_m(O_G^y\cap \Omega_m);
\]
\item[(iv)] \(\varphi_m\) and \(\varphi_m^{-1}\) are bi-Lipschitz on \(\overline{\Omega_m}\).
\end{enumerate}

For \(u\in C_G^\infty(M)\), property (iii) implies that \(u\circ \varphi_m^{-1}\) is independent of the \(U_m\)-variable.
Hence there exists a smooth function
\[
\widetilde u_m\in C^\infty(V_m)
\]
such that
\begin{equation}\label{eq3.1}
u\circ \varphi_m^{-1}(x,z)=\widetilde u_m(z)
\qquad \text{for all } (x,z)\in U_m\times V_m.
\end{equation}

\begin{lemma}\label{lem3.3}
Let \(U\subset \R^\ell\) be a bounded open set with nonempty interior, and let \(d\ge 1\).
Fix \(N>\ell\), and define
\[
J(\xi)=\iint_{U\times U}\frac{dx\,dy}{\bigl(|x-y|+|\xi|\bigr)^N},
\qquad \xi\in \R^d\setminus\{0\}.
\]
Then there exist positive constants \(c\) and \(C\), depending only on \(U\) and \(N\), such that
\[
c\,|\xi|^{-(N-\ell)}\le J(\xi)\le C\,|\xi|^{-(N-\ell)}
\qquad \text{for all } \xi\in (\overline V-\overline V)\setminus\{0\},
\]
where \(V\subset \R^d\) is any fixed bounded set.
\end{lemma}

\begin{proof}
Since \(U\) is bounded, there exists \(R>0\) such that \(U-U\subset B_R(0)\).
Using the change of variables \(\eta=x-y\), we obtain
\[
J(\xi)=\int_{U-U}\frac{m_U(\eta)}{(|\eta|+|\xi|)^N}\,d\eta,
\]
where
\[
m_U(\eta)=\bigl|\{(x,y)\in U\times U:x-y=\eta\}\bigr|
\]
is bounded above by \(|U|\).
Therefore,
\[
J(\xi)\le C\int_{B_R(0)}\frac{d\eta}{(|\eta|+|\xi|)^N}
\le C |\xi|^{\ell-N}.
\]

For the lower bound, choose a ball \(B_\rho(a)\subset U\).
Then, for \(|\eta|<\rho/2\), the overlap measure of \(B_\rho(a)\) and \(B_\rho(a)-\eta\) is bounded below by a positive constant.
Hence
\[
J(\xi)\ge c\int_{B_{\rho/2}(0)}\frac{d\eta}{(|\eta|+|\xi|)^N}.
\]
If \(|\xi|<\rho/2\), restricting further to \(|\eta|<|\xi|\) gives
\[
J(\xi)\ge c\,|\xi|^\ell (2|\xi|)^{-N}=c\,|\xi|^{\ell-N}.
\]
If \(|\xi|\ge \rho/2\), then \(|\xi|^{\ell-N}\) is bounded above on the bounded set \((\overline V-\overline V)\setminus\{0\}\), while \(J(\xi)\) remains bounded below by a positive constant.
This proves the claim.
\end{proof}

We now compare the local \(L^p\) norms and local nonlocal energies of \(u\) and \(\widetilde u_m\).

\begin{lemma}\label{lem3.4}
For every \(m\) and every \(p\ge 1\), there exist constants \(A_m,B_m>0\) such that for every \(u\in C_G^\infty(M)\),
\[
A_m \int_{V_m} |\widetilde u_m(z)|^p\,dz
\le
\int_{\Omega_m} |u|^p\,d\mu
\le
B_m \int_{V_m} |\widetilde u_m(z)|^p\,dz.
\]
Moreover, there exist constants \(c_m,C_m>0\) such that
\[
c_m
\iint_{V_m\times V_m}
\frac{|\widetilde u_m(z)-\widetilde u_m(w)|^q}{|z-w|^{d_m+sq}}
\,dz\,dw
\le
\iint_{\Omega_m\times\Omega_m}
|u(x)-u(y)|^q K_q^s(x,y)\,d\mu(x)\,d\mu(y)
\]
and
\[
\iint_{\Omega_m\times\Omega_m}
|u(x)-u(y)|^q K_q^s(x,y)\,d\mu(x)\,d\mu(y)
\le
C_m
\iint_{V_m\times V_m}
\frac{|\widetilde u_m(z)-\widetilde u_m(w)|^q}{|z-w|^{d_m+sq}}
\,dz\,dw.
\]
\end{lemma}

\begin{proof}
The \(L^p\) comparison follows directly from \eqref{eq3.1}, the boundedness of \(U_m\), and the equivalence between \(d\mu\) and the Euclidean measure in the chart.

For the nonlocal term, by Remark 3.4 in \cite{Tan2025SharpFS}, there exist constants \(0<c_{M,s,q}\le C_{M,s,q}\) such that
\[
\frac{c_{M,s,q}}{d_g(x,y)^{n+sq}}
\le
K_q^s(x,y)
\le
\frac{C_{M,s,q}}{d_g(x,y)^{n+sq}},
\qquad x\ne y.
\]
Since \(\varphi_m\) and \(\varphi_m^{-1}\) are bi-Lipschitz on \(\overline{\Omega_m}\), there exists \(L_m\ge 1\) such that
\[
L_m^{-1}\bigl(|x-y|+|z-w|\bigr)
\le
d_g\bigl(\varphi_m^{-1}(x,z),\varphi_m^{-1}(y,w)\bigr)
\le
L_m\bigl(|x-y|+|z-w|\bigr)
\]
for all \((x,z),(y,w)\in U_m\times V_m\).
Using \eqref{eq3.1}, we obtain
\[
\iint_{\Omega_m\times\Omega_m}
|u(x)-u(y)|^q K_q^s(x,y)\,d\mu(x)\,d\mu(y)
\asymp
\iint_{V_m\times V_m}
|\widetilde u_m(z)-\widetilde u_m(w)|^q J_m(z-w)\,dz\,dw,
\]
where
\[
J_m(\xi)=\iint_{U_m\times U_m}\frac{dx\,dy}{\bigl(|x-y|+|\xi|\bigr)^{n+sq}}.
\]
Applying Lemma \ref{lem3.3} with \(\ell=k_m\) and \(N=n+sq\), we obtain
\[
J_m(\xi)\asymp |\xi|^{-(d_m+sq)}.
\]
This yields the desired two-sided estimate.
\end{proof}

We now recall the Euclidean fractional Sobolev and Morrey embeddings in the relevant dimension.

\begin{lemma}\label{lem3.5}
Let \(\Omega\subset \R^d\) be a bounded Lipschitz domain, let \(q\in[1,\infty)\), and let \(s\in(0,1)\).

\begin{enumerate}
\item[(i)] If \(sq<d\), then
\[
W^{s,q}(\Omega)\hookrightarrow L^p(\Omega)
\]
continuously for every
\[
1\le p\le \frac{dq}{d-sq}.
\]

\item[(ii)] If \(sq>d\), then
\[
W^{s,q}(\Omega)\hookrightarrow C^{0,\alpha}(\overline{\Omega})
\]
continuously with
\[
\alpha=s-\frac{d}{q}.
\]
In particular,
\[
W^{s,q}(\Omega)\hookrightarrow L^p(\Omega)
\]
continuously for every \(p\ge 1\).
\end{enumerate}
\end{lemma}

\begin{proof}
This is standard.
See \cite[Theorems 6.7 and 8.2]{MR2944369}.
\end{proof}

\begin{lemma}\label{lem3.6}
Let \(\Omega\subset \R^d\) be a bounded Lipschitz domain, let \(s\in(0,1)\), and let \(q\in[1,\infty)\) satisfy
\[
sq=d.
\]
Then, for every \(p\in[1,\infty)\), the embedding
\[
W^{s,q}(\Omega)\hookrightarrow L^p(\Omega)
\]
is continuous and compact.
\end{lemma}

\begin{proof}
Fix \(p\in[1,\infty)\).
Choose \(t\in(0,s)\) sufficiently close to \(s\) such that
\[
p<\frac{dq}{d-tq}.
\]
This is possible because
\[
\frac{dq}{d-tq}\to+\infty
\qquad \text{as } t\uparrow s,
\]
and \(d-sq=0\).

We first show that
\[
W^{s,q}(\Omega)\hookrightarrow W^{t,q}(\Omega)
\]
continuously.
Indeed, for every \(u\in W^{s,q}(\Omega)\),
\[
[u]_{W^{t,q}(\Omega)}^q
=
\iint_{\Omega\times\Omega}
\frac{|u(x)-u(y)|^q}{|x-y|^{d+tq}}\,dx\,dy.
\]
Since \(\Omega\) is bounded,
\[
|x-y|^{(s-t)q}\le (\diam \Omega)^{(s-t)q}
\qquad \text{for all } x,y\in \Omega.
\]
Hence
\[
\begin{aligned}
[u]_{W^{t,q}(\Omega)}^q
&=
\iint_{\Omega\times\Omega}
\frac{|u(x)-u(y)|^q}{|x-y|^{d+sq}}
\,|x-y|^{(s-t)q}\,dx\,dy \\
&\le
(\diam \Omega)^{(s-t)q}
\iint_{\Omega\times\Omega}
\frac{|u(x)-u(y)|^q}{|x-y|^{d+sq}}\,dx\,dy \\
&=
(\diam \Omega)^{(s-t)q}[u]_{W^{s,q}(\Omega)}^q.
\end{aligned}
\]
Therefore,
\[
\|u\|_{W^{t,q}(\Omega)}
\le
C\|u\|_{W^{s,q}(\Omega)}.
\]

Since \(tq<d\), Lemma~\ref{lem3.5} yields the continuous embedding
\[
W^{t,q}(\Omega)\hookrightarrow L^p(\Omega),
\]
because
\[
p<\frac{dq}{d-tq}.
\]
Combining the two embeddings, we obtain the continuity of
\[
W^{s,q}(\Omega)\hookrightarrow L^p(\Omega).
\]

We now prove compactness.
Let \((u_j)\) be a bounded sequence in \(W^{s,q}(\Omega)\).
By the continuous embedding just proved, \((u_j)\) is bounded in \(W^{t,q}(\Omega)\).
Since \(tq<d\) and
\[
p<\frac{dq}{d-tq},
\]
the fractional Rellich theorem, see \cite[Corollary 7.2]{MR2944369}, implies that \((u_j)\) is relatively compact in \(L^p(\Omega)\).
Hence
\[
W^{s,q}(\Omega)\hookrightarrow L^p(\Omega)
\]
is compact.
\end{proof}

\begin{lemma}\label{lem3.7}
Assume that \(d_m=n-k_m<sq\).
Then for every \(p\ge 1\), there exists a constant \(C_m(p)>0\) such that for every \(u\in C_G^\infty(M)\),
\[
\|u\|_{L^p(\Omega_m)}
\le
C_m(p)\Bigl(
\|u\|_{L^q(\Omega_m)}
+
\Bigl(
\iint_{\Omega_m\times\Omega_m}
|u(x)-u(y)|^q K_q^s(x,y)\,d\mu(x)\,d\mu(y)
\Bigr)^{1/q}
\Bigr).
\]
\end{lemma}

\begin{proof}
By Lemma \ref{lem3.4}, it is enough to estimate \(\widetilde u_m\) on \(V_m\).
Since \(d_m<sq\), Lemma \ref{lem3.5}(ii) gives
\[
\|\widetilde u_m\|_{L^p(V_m)}
\le
C\|\widetilde u_m\|_{W^{s,q}(V_m)}.
\]
Using Lemma \ref{lem3.4} again to control the \(L^q\) norm and the fractional seminorm of \(\widetilde u_m\) by the corresponding quantities of \(u\) on \(\Omega_m\), we obtain the desired result.
\end{proof}

\begin{lemma}\label{lem3.8}
Assume that \(d_m=n-k_m>sq\).
Then for every
\[
1\le p\le \frac{d_m q}{d_m-sq},
\]
there exists a constant \(C_m(p)>0\) such that for every \(u\in C_G^\infty(M)\),
\[
\|u\|_{L^p(\Omega_m)}
\le
C_m(p)\Bigl(
\|u\|_{L^q(\Omega_m)}
+
\Bigl(
\iint_{\Omega_m\times\Omega_m}
|u(x)-u(y)|^q K_q^s(x,y)\,d\mu(x)\,d\mu(y)
\Bigr)^{1/q}
\Bigr).
\]
\end{lemma}

\begin{proof}
By Lemma \ref{lem3.5}(i) in dimension \(d_m\),
\[
\|\widetilde u_m\|_{L^p(V_m)}
\le
C\|\widetilde u_m\|_{W^{s,q}(V_m)}
\qquad
\text{for }
1\le p\le \frac{d_m q}{d_m-sq}.
\]
The conclusion now follows from Lemma \ref{lem3.4}.
\end{proof}

We can now globalize the local estimates.

\begin{lemma}\label{lem3.9}
The following continuous embeddings hold:
\begin{enumerate}
\item[(i)] if \(n-k<sq\), then
\[
W_G^{s,q}(M)\hookrightarrow L^p(M)
\qquad \text{for every } p\ge 1;
\]
\item[(ii)] if \(n-k>sq\), then
\[
W_G^{s,q}(M)\hookrightarrow L^p(M)
\qquad \text{for every }
1\le p\le \frac{(n-k)q}{n-k-sq}.
\]
\end{enumerate}
\end{lemma}

\begin{proof}
By Lemma \ref{lem2.2}, \(C_G^\infty(M)\) is dense in \(W_G^{s,q}(M)\).
It is therefore enough to prove the estimate for \(u\in C_G^\infty(M)\).
Since the covering is finite,
\[
\|u\|_{L^p(M)}
\le
\sum_{m=1}^N \|u\|_{L^p(\Omega_m)}.
\]
Moreover,
\[
\|u\|_{L^q(\Omega_m)}\le \|u\|_{L^q(M)},
\]
and
\[
\iint_{\Omega_m\times\Omega_m}
|u(x)-u(y)|^q K_q^s(x,y)\,d\mu(x)\,d\mu(y)
\le
\iint_{M\times M}
|u(x)-u(y)|^q K_q^s(x,y)\,d\mu(x)\,d\mu(y).
\]

If \(n-k<sq\), then \(d_m=n-k_m\le n-k<sq\) for every \(m\), so Lemma \ref{lem3.7} applies to each chart.
Summing the local estimates, we obtain
\[
\|u\|_{L^p(M)}
\le
C\Bigl(
\|u\|_{L^q(M)}
+
[u]_{W^{s,q}(M)}
\Bigr)
\]
for every \(p\ge 1\).

Assume now that \(n-k>sq\), and let
\[
1\le p\le \frac{(n-k)q}{n-k-sq}.
\]
If \(d_m<sq\), then Lemma~\ref{lem3.7} applies.
If \(d_m=sq\), then Lemma~\ref{lem3.6} applies.
If \(d_m>sq\), then Lemma~\ref{lem3.8} applies.
Moreover, since \(d_m\le n-k\) and the function
\[
t\mapsto \frac{tq}{t-sq}
\]
is decreasing on \((sq,\infty)\), one has
\[
p\le \frac{(n-k)q}{n-k-sq}\le \frac{d_m q}{d_m-sq}
\qquad \text{whenever } d_m>sq.
\]
Thus every chart yields the required \(L^p\) bound, and summing once again over \(m\) gives the global estimate.
By density, the conclusion extends to all \(u\in W_G^{s,q}(M)\).
\end{proof}

We next prove compactness.

\begin{lemma}\label{lem3.10}
Assume that \(n-k<sq\), and set
\[
\alpha=s-\frac{n-k}{q}>0.
\]
Then
\[
W_G^{s,q}(M)\hookrightarrow C^{0,\alpha}(M)
\]
continuously.
\end{lemma}

\begin{proof}
Fix \(u\in C_G^\infty(M)\).
For every \(m\), since \(d_m\le n-k<sq\), Lemma \ref{lem3.5}(ii) gives
\[
\|\widetilde u_m\|_{C^{0,\alpha_m}(\overline{V_m})}
\le
C_m \|\widetilde u_m\|_{W^{s,q}(V_m)},
\qquad
\alpha_m=s-\frac{d_m}{q}\ge \alpha.
\]
Because \(V_m\) is bounded, the \(C^{0,\alpha_m}\)-norm controls the \(C^{0,\alpha}\)-norm.
Combining this with Lemma \ref{lem3.4}, we obtain
\[
\|u\|_{C^{0,\alpha}(\Omega_m)}
\le
C_m\Bigl(
\|u\|_{L^q(\Omega_m)}
+
[u]_{W^{s,q}(M)}
\Bigr).
\]
Since the covering is finite, these local estimates imply
\[
\|u\|_{C^{0,\alpha}(M)}
\le
C\Bigl(
\|u\|_{L^q(M)}
+
[u]_{W^{s,q}(M)}
\Bigr).
\]
By density, the estimate extends to \(W_G^{s,q}(M)\).
\end{proof}

\begin{lemma}\label{lem3.11}
The following compact embeddings hold:
\begin{enumerate}
\item[(i)] if \(n-k<sq\), then
\[
W_G^{s,q}(M)\hookrightarrow L^p(M)
\qquad \text{compactly for every } p\ge 1;
\]
\item[(ii)] if \(n-k>sq\), then
\[
W_G^{s,q}(M)\hookrightarrow L^p(M)
\qquad \text{compactly for every }
1\le p<\frac{(n-k)q}{n-k-sq}.
\]
\end{enumerate}
\end{lemma}

\begin{proof}
We first prove (i).
Let \((u_j)\) be a bounded sequence in \(W_G^{s,q}(M)\).
By Lemma \ref{lem3.10}, \((u_j)\) is bounded in \(C^{0,\alpha}(M)\), hence uniformly bounded and equicontinuous on the compact manifold \(M\).
By the Arzel\`a--Ascoli theorem, there exist a subsequence, still denoted by \((u_j)\), and a function \(u\in C(M)\) such that
\[
u_j\to u \qquad \text{uniformly on } M.
\]
Therefore,
\[
\|u_j-u\|_{L^p(M)}
\le
\Vol(M)^{1/p}\|u_j-u\|_{L^\infty(M)}
\to 0.
\]
Thus the embedding into \(L^p(M)\) is compact for every \(p\ge 1\).

We now prove (ii).
Fix
\[
1\le p<\frac{(n-k)q}{n-k-sq},
\]
and let \((u_j)\) be a bounded sequence in \(W_G^{s,q}(M)\).
For each \(m\), let \(\widetilde u_{j,m}\) be the function on \(V_m\) defined by
\[
u_j\circ \varphi_m^{-1}(x,z)=\widetilde u_{j,m}(z).
\]
By Lemma \ref{lem3.4}, the sequence \((\widetilde u_{j,m})_j\) is bounded in \(W^{s,q}(V_m)\).

If \(d_m>sq\), then
\[
p<\frac{(n-k)q}{n-k-sq}\le \frac{d_m q}{d_m-sq},
\]
so the Euclidean compact embedding
\[
W^{s,q}(V_m)\hookrightarrow L^p(V_m)
\]
holds.
If \(d_m<sq\), then \(W^{s,q}(V_m)\) embeds continuously into \(C^{0,\beta}(\overline{V_m})\), and hence compactly into \(L^p(V_m)\).
If \(d_m=sq\), then Lemma~\ref{lem3.6} yields the compact embedding \(W^{s,q}(V_m)\hookrightarrow L^p(V_m)\) for every finite \(p\).
Therefore, for every \(m\), the sequence \((\widetilde u_{j,m})_j\) admits a subsequence converging in \(L^p(V_m)\).

Since there are only finitely many indices \(m\), a diagonal argument yields a single subsequence, still denoted by \((u_j)\), such that for every \(m\),
\[
\widetilde u_{j,m}\to \widetilde u_m
\qquad \text{in } L^p(V_m).
\]
Using the \(L^p\) comparison in Lemma \ref{lem3.4}, we obtain
\[
\|u_j-u_\ell\|_{L^p(\Omega_m)}
\le
C_m \|\widetilde u_{j,m}-\widetilde u_{\ell,m}\|_{L^p(V_m)}
\to 0
\qquad \text{as } j,\ell\to\infty.
\]
Summing over the finite covering, we get
\[
\|u_j-u_\ell\|_{L^p(M)}
\le
\sum_{m=1}^N \|u_j-u_\ell\|_{L^p(\Omega_m)}
\to 0.
\]
Hence \((u_j)\) is a Cauchy sequence in \(L^p(M)\), and therefore converges in \(L^p(M)\).
This proves the compactness.
\end{proof}

\begin{proof}[Proof of Theorem~\ref{thm1.1}]
Assertion (i) follows from Lemmas~\ref{lem3.9}(i) and \ref{lem3.11}(i), while assertion (ii) follows from Lemmas~\ref{lem3.9}(ii) and \ref{lem3.11}(ii).
\end{proof}

\section{Proof of Theorem~\ref{thm1.2}}

Throughout this section, we assume that
\[
n-k>sq,
\]
and we set
\[
\widetilde q_s^*=\frac{(n-k)q}{n-k-sq}.
\]
This is precisely the range in which the improved critical exponent is finite.

\begin{lemma}\label{lem4.1}
There exists a constant \(C>0\) such that
\[
\|u-u_M\|_{L^q(M)}
\le
C [u]_{W^{s,q}(M)}
\qquad
\text{for every } u\in W_G^{s,q}(M),
\]
where
\[
u_M=\frac{1}{\Vol(M)}\int_M u\,d\mu .
\]
\end{lemma}

\begin{proof}
We argue by contradiction.
Suppose that the conclusion is false.
Then there exists a sequence \((u_j)\subset W_G^{s,q}(M)\) such that
\[
(u_j)_M=0,
\qquad
\|u_j\|_{L^q(M)}=1,
\qquad
[u_j]_{W^{s,q}(M)}\to 0
\quad \text{as } j\to\infty.
\]
By Lemma~\ref{lem3.11}, since
\[
q<\widetilde q_s^*=\frac{(n-k)q}{n-k-sq},
\]
the embedding
\[
W_G^{s,q}(M)\hookrightarrow L^q(M)
\]
is compact.
Hence, up to a subsequence,
\[
u_j\to u
\qquad
\text{strongly in } L^q(M),
\]
for some \(u\in L^q(M)\).
In particular,
\[
\|u\|_{L^q(M)}=1
\qquad \text{and} \qquad
u_M=0.
\]

Passing to a further subsequence if necessary, we may also assume that
\[
u_j(x)\to u(x)
\qquad
\text{for a.e. } x\in M.
\]
Therefore,
\[
u_j(x)-u_j(y)\to u(x)-u(y)
\qquad
\text{for a.e. } (x,y)\in M\times M.
\]
By Fatou's lemma,
\[
[u]_{W^{s,q}(M)}^q
=
\iint_{M\times M}|u(x)-u(y)|^q K_q^s(x,y)\,d\mu(x)\,d\mu(y)
\le
\liminf_{j\to\infty}[u_j]_{W^{s,q}(M)}^q
=
0.
\]
Hence \(u(x)=u(y)\) for a.e. \((x,y)\in M\times M\), so \(u\) is almost everywhere equal to a constant.
Since \(u_M=0\), this constant must be \(0\).
Thus \(u=0\) a.e. on \(M\), which contradicts \(\|u\|_{L^q(M)}=1\).
The proof is complete.
\end{proof}

\begin{lemma}\label{lem4.2}
There exists a constant \(A_0>0\) such that
\[
\|u-u_M\|_{L^{\widetilde q_s^*}(M)}
\le
A_0 [u]_{W^{s,q}(M)}
\qquad
\text{for every } u\in W_G^{s,q}(M).
\]
\end{lemma}

\begin{proof}
Set
\[
v=u-u_M.
\]
Then \(v\in W_G^{s,q}(M)\), \(v_M=0\), and
\[
[v]_{W^{s,q}(M)}=[u]_{W^{s,q}(M)}.
\]
By Lemma~\ref{lem3.9}, the continuous embedding
\[
W_G^{s,q}(M)\hookrightarrow L^{\widetilde q_s^*}(M)
\]
holds.
Therefore, there exists \(C_0>0\) such that
\[
\|v\|_{L^{\widetilde q_s^*}(M)}
\le
C_0\Bigl(\|v\|_{L^q(M)}+[v]_{W^{s,q}(M)}\Bigr).
\]
Applying Lemma~\ref{lem4.1}, we obtain
\[
\|v\|_{L^{\widetilde q_s^*}(M)}
\le
C_0(C+1)[u]_{W^{s,q}(M)}.
\]
This proves the lemma.
\end{proof}

\begin{proof}[Proof of Theorem~\ref{thm1.2}]
We first prove the lower bound
\[
\beta_{q,G}(M)\ge \Vol(M)^{-s/(n-k)}.
\]
Indeed, assume that
\[
\|u\|_{L^{\widetilde q_s^*}(M)}
\le
A [u]_{W^{s,q}(M)}
+
B\|u\|_{L^q(M)}
\]
holds for all \(u\in W_G^{s,q}(M)\).
Testing this inequality with the constant function \(u\equiv 1\), for which \([u]_{W^{s,q}(M)}=0\), we get
\[
\Vol(M)^{1/\widetilde q_s^*}
\le
B\,\Vol(M)^{1/q}.
\]
Hence
\[
B\ge \Vol(M)^{1/\widetilde q_s^*-1/q}
=
\Vol(M)^{-s/(n-k)}.
\]
It follows that
\[
\beta_{q,G}(M)\ge \Vol(M)^{-s/(n-k)}.
\]

We now prove the reverse inequality.
By Lemma~\ref{lem4.2}, for every \(u\in W_G^{s,q}(M)\),
\[
\|u-u_M\|_{L^{\widetilde q_s^*}(M)}
\le
A_0 [u]_{W^{s,q}(M)}.
\]
Therefore, by the triangle inequality,
\begin{equation}\label{eq4.1}
\|u\|_{L^{\widetilde q_s^*}(M)}
\le
A_0 [u]_{W^{s,q}(M)}
+
\|u_M\|_{L^{\widetilde q_s^*}(M)}.
\end{equation}
Since \(u_M\) is constant, we have
\[
\|u_M\|_{L^{\widetilde q_s^*}(M)}
=
\Vol(M)^{1/\widetilde q_s^*}\,|u_M|.
\]
Moreover, Hölder's inequality gives
\[
|u_M|
=
\frac{1}{\Vol(M)}\left|\int_M u\,d\mu\right|
\le
\Vol(M)^{-1/q}\|u\|_{L^q(M)}.
\]
Consequently,
\[
\|u_M\|_{L^{\widetilde q_s^*}(M)}
\le
\Vol(M)^{1/\widetilde q_s^*-1/q}\|u\|_{L^q(M)}
=
\Vol(M)^{-s/(n-k)}\|u\|_{L^q(M)}.
\]
Substituting this estimate into \eqref{eq4.1}, we obtain
\[
\|u\|_{L^{\widetilde q_s^*}(M)}
\le
A_0 [u]_{W^{s,q}(M)}
+
\Vol(M)^{-s/(n-k)}\|u\|_{L^q(M)}.
\]
Hence
\[
\beta_{q,G}(M)\le \Vol(M)^{-s/(n-k)}.
\]
Combining the two bounds, we conclude that
\[
\beta_{q,G}(M)=\Vol(M)^{-s/(n-k)}.
\]
At the same time, the preceding estimate shows that there exists \(A\in \R\) such that
\[
\|u\|_{L^{\widetilde q_s^*}(M)}
\le
A [u]_{W^{s,q}(M)}
+
\Vol(M)^{-s/(n-k)}\|u\|_{L^q(M)}
\]
for all \(u\in W_G^{s,q}(M)\).
This completes the proof.
\end{proof}

\section{Proof of Theorem~\ref{thm1.4}}

Throughout this section, we assume that
\[
s\in(0,1),\qquad q\in[1,\infty),\qquad sq<n,
\]
and we set
\[
q_s^*=\frac{nq}{n-sq}.
\]

The next lemma provides the key local estimate near infinite orbits.
It shows that, on a neighborhood where the connected component \(G_0\) of the identity has only infinite orbits, the Sobolev coefficient can be made arbitrarily small.

\begin{lemma}\label{lem5.1}
Let \((M,g)\) be a closed Riemannian \(n\)-manifold, let \(G\subset \Isom_g(M)\) be compact, and let \(G_0\) be the identity component of \(G\).
Let \(O\subset M\) be a compact set such that
\begin{enumerate}
\item \(O\) is invariant under the action of \(G_0\);
\item for every \(x\in O\), the orbit \(O_{G_0}^x\) is infinite.
\end{enumerate}
Then there exists \(\delta_0>0\) such that, for every \(\varepsilon>0\), there exists \(B_\varepsilon>0\) with the following property: for every \(u\in C_G^\infty(M)\) satisfying
\[
\supp u\subset O_{\delta_0}
=
\{y\in M:d_g(y,O)<\delta_0\},
\]
one has
\begin{equation}\label{eq5.1}
\|u\|_{L^{q_s^*}(M)}^q
\le
\varepsilon [u]_{W^{s,q}(M)}^q
+
B_\varepsilon \|u\|_{L^q(M)}^q.
\end{equation}
\end{lemma}

\begin{proof}
Since all \(G_0\)-orbits in \(O\) are infinite, Proposition~\ref{pro3.1} implies that
\[
k_O=\min_{x\in O}\dim O_{G_0}^x \ge 1.
\]
Choose \(\delta_0>0\) so small that \(O_{\delta_0}\) is covered by finitely many charts of the type used in the proof of Theorem~\ref{thm1.1}, with respect to the group \(G_0\), and such that, in each of these charts, the orbit dimension is at least \(k_O\).

Let
\[
X_O=
\bigl\{
u\in W_{G_0}^{s,q}(M): \supp u\subset O_{\delta_0}
\bigr\}.
\]
Repeating the proof of Theorem~\ref{thm1.1} on this finite covering, one finds an exponent
\[
r>\frac{nq}{n-sq}=q_s^*
\]
such that
\[
X_O\hookrightarrow L^r(M)
\]
continuously, and therefore
\[
X_O\hookrightarrow L^{q_s^*}(M)
\]
compactly.

We now apply Ehrling's lemma to the compact embedding
\[
X_O\hookrightarrow L^{q_s^*}(M)
\]
and the continuous embedding
\[
L^{q_s^*}(M)\hookrightarrow L^q(M).
\]
Thus, for every \(\theta>0\), there exists \(C_\theta>0\) such that
\[
\|u\|_{L^{q_s^*}(M)}
\le
\theta \|u\|_{W^{s,q}(M)}
+
C_\theta \|u\|_{L^q(M)}
\qquad \text{for all } u\in X_O.
\]
Since
\[
\|u\|_{W^{s,q}(M)}
=
\|u\|_{L^q(M)}+[u]_{W^{s,q}(M)},
\]
we obtain
\[
\|u\|_{L^{q_s^*}(M)}
\le
\theta [u]_{W^{s,q}(M)}
+
(\theta+C_\theta)\|u\|_{L^q(M)}.
\]
Raising both sides to the power \(q\) and using
\[
(a+b)^q\le 2^{q-1}(a^q+b^q),
\]
we get
\[
\|u\|_{L^{q_s^*}(M)}^q
\le
2^{q-1}\theta^q [u]_{W^{s,q}(M)}^q
+
2^{q-1}(\theta+C_\theta)^q \|u\|_{L^q(M)}^q.
\]
Choosing \(\theta>0\) so small that \(2^{q-1}\theta^q\le \varepsilon\), we obtain \eqref{eq5.1}.
Since every \(u\in C_G^\infty(M)\) is, in particular, \(G_0\)-invariant, the proof is complete.
\end{proof}

We also need a localization estimate for multiplication by smooth cut-off functions.

\begin{lemma}\label{lem5.2}
Let \(\psi\in C^1(M)\) be nonnegative, and assume that \(\psi^{1/q}\in C^1(M)\).
Then, for every \(\tau\in(0,1)\), there exists \(C_{\psi,\tau}>0\) such that, for every \(u\in W^{s,q}(M)\),
\begin{equation}\label{eq5.2}
[\psi^{1/q}u]_{W^{s,q}(M)}^q
\le
(1+\tau)^{q-1}
\iint_{M\times M}\psi(x)|u(x)-u(y)|^q K_q^s(x,y)\,d\mu(x)\,d\mu(y)
+
C_{\psi,\tau}\|u\|_{L^q(M)}^q.
\end{equation}
\end{lemma}

\begin{proof}
For \(x,y\in M\),
\[
\psi(x)^{1/q}u(x)-\psi(y)^{1/q}u(y)
=
\psi(x)^{1/q}\bigl(u(x)-u(y)\bigr)
+
\bigl(\psi(x)^{1/q}-\psi(y)^{1/q}\bigr)u(y).
\]
Hence, for every \(\tau\in(0,1)\),
\[
|\psi(x)^{1/q}u(x)-\psi(y)^{1/q}u(y)|^q
\le
(1+\tau)^{q-1}\psi(x)|u(x)-u(y)|^q
+
\Bigl(1+\frac{1}{\tau}\Bigr)^{q-1}
|\psi(x)^{1/q}-\psi(y)^{1/q}|^q |u(y)|^q.
\]
After integrating against \(K_q^s(x,y)\,d\mu(x)\,d\mu(y)\), it remains to estimate the second term.

Since \(\psi^{1/q}\in C^1(M)\), there exists \(L_\psi>0\) such that
\[
|\psi(x)^{1/q}-\psi(y)^{1/q}|
\le
L_\psi\, d_g(x,y)
\qquad \text{for all } x,y\in M.
\]
By Remark 3.4 in \cite{Tan2025SharpFS}, there exists a constant \(C_{M,s,q}>0\) such that
\[
K_q^s(x,y)\le \frac{C_{M,s,q}}{d_g(x,y)^{n+sq}},
\qquad x\ne y.
\]
Hence
\[
|\psi(x)^{1/q}-\psi(y)^{1/q}|^q K_q^s(x,y)
\le
C\, d_g(x,y)^{q}\cdot d_g(x,y)^{-n-sq}
=
C\, d_g(x,y)^{-n+(1-s)q}.
\]
Because \((1-s)q>0\) and \(M\) is compact, the function
\[
x\mapsto d_g(x,y)^{-n+(1-s)q}
\]
is integrable uniformly in \(y\).
Therefore,
\[
\sup_{y\in M}\int_M |\psi(x)^{1/q}-\psi(y)^{1/q}|^q K_q^s(x,y)\,d\mu(x)<\infty.
\]
The desired estimate now follows from Fubini's theorem.
\end{proof}

\begin{proof}[Proof of Theorem~\ref{thm1.4}]
Let \(G_0\) be the identity component of \(G\).

For each \(x\in M\), we construct a \(G\)-stable neighborhood \(U_x\) as follows.

If \(O_G^x\) is finite, write
\[
O_G^x=\{x_1,\dots,x_m\},
\qquad x_1=x,
\]
and choose \(\delta_x\in(0,\inj(M,g))\) so small that the geodesic balls
\[
B_{x_j}(2\delta_x),
\qquad j=1,\dots,m,
\]
are pairwise disjoint.
Set
\[
U_x=\bigcup_{j=1}^m B_{x_j}(\delta_x).
\]

Assume now that \(O_G^x\) is infinite.
By Proposition~\ref{pro3.1}, it is a compact submanifold of positive dimension.
Let
\[
O_G^x=O_1\cup\cdots\cup O_m
\]
be the decomposition into connected components, with \(x\in O_1\).
Since \(G_0\) is connected, each \(O_j\) is \(G_0\)-invariant.
Moreover, every \(G_0\)-orbit in \(O_1\cup\cdots\cup O_m\) is infinite.
Choose \(\delta_x>0\) so small that
\begin{enumerate}
\item \(\delta_x\) is smaller than the radius given by Lemma~\ref{lem5.1} with \(O=O_1\cup\cdots\cup O_m\);
\item the tubular neighborhoods
\[
B_j^{\delta_x}
=
\{y\in M:d_g(y,O_j)<\delta_x\},
\qquad j=1,\dots,m,
\]
are pairwise disjoint;
\item for each \(j\), the squared distance to \(O_j\) is smooth on \(B_j^{2\delta_x}\).
\end{enumerate}
Set
\[
U_x=\bigcup_{j=1}^m B_j^{\delta_x}.
\]

By compactness of \(M\), there exist finitely many points \(x_1,\dots,x_N\in M\) such that
\[
M=\bigcup_{i=1}^N U_{x_i}.
\]
For simplicity, write
\[
U_i=U_{x_i}.
\]

Let
\[
I_1=\{\,i\in\{1,\dots,N\}: \Card O_G^{x_i}<+\infty\,\},
\qquad
I_2=\{1,\dots,N\}\setminus I_1.
\]

For each \(i\in I_1\), write
\[
U_i=U_{i1}\cup\cdots\cup U_{im_i},
\]
where the \(U_{ij}\) are the pairwise disjoint balls centered at the points of the finite orbit \(O_G^{x_i}\), so that
\[
m_i=\Card O_G^{x_i}.
\]
Choose smooth nonnegative functions \(\alpha_{ij}\in C^\infty(M)\) such that
\[
\supp \alpha_{ij}\subset U_{ij},
\]
and such that, for each fixed \(i\), the family \(\{\alpha_{ij}\}_{j=1}^{m_i}\) is permuted by the action of \(G\).

For each \(i\in I_2\), choose a smooth nonnegative function \(\beta_i\in C^\infty(M)\) such that
\[
\supp \beta_i\subset U_i,
\]
and \(\beta_i\) is \(G\)-invariant.
This is possible because \(U_i\) is \(G\)-stable: one may start from any smooth cut-off supported in \(U_i\) and average it over \(G\).

We choose these functions so that
\[
\sum_{i\in I_1}\sum_{j=1}^{m_i}\alpha_{ij}
+
\sum_{i\in I_2}\beta_i
>0
\qquad \text{on } M.
\]
Set
\[
\Lambda
=
\sum_{i\in I_1}\sum_{j=1}^{m_i}\alpha_{ij}^{[q]+1}
+
\sum_{i\in I_2}\beta_i^{[q]+1},
\]
and define
\[
\eta_{ij}
=
\frac{\alpha_{ij}^{[q]+1}}{\Lambda},
\qquad i\in I_1,\ j=1,\dots,m_i,
\]
and
\[
\zeta_i
=
\frac{\beta_i^{[q]+1}}{\Lambda},
\qquad i\in I_2.
\]
Then
\[
\sum_{i\in I_1}\sum_{j=1}^{m_i}\eta_{ij}
+
\sum_{i\in I_2}\zeta_i
=1
\qquad \text{on } M,
\]
and all these functions are smooth and nonnegative.
Moreover,
\[
\eta_{ij}^{1/q}\in C^1(M),
\qquad
\zeta_i^{1/q}\in C^1(M).
\]
For each fixed \(i\in I_1\), the supports of \(\eta_{ij}\) and \(\eta_{ij'}\) are disjoint whenever \(j\ne j'\), and the family \(\{\eta_{ij}\}_{j=1}^{m_i}\) is still permuted by the action of \(G\).
In addition, each \(\zeta_i\) is \(G\)-invariant.

Now let \(u\in C_G^\infty(M)\).
Since
\[
|u|^q
=
\sum_{i\in I_1}\sum_{j=1}^{m_i}\eta_{ij}|u|^q
+
\sum_{i\in I_2}\zeta_i |u|^q,
\]
and \(q_s^*/q>1\), Minkowski's inequality yields
\begin{equation}\label{eq5.3}
\|u\|_{L^{q_s^*}(M)}^q
\le
\sum_{i\in I_1}
\left(
\int_M
\Bigl(\sum_{j=1}^{m_i}\eta_{ij}|u|^q\Bigr)^{q_s^*/q}\,d\mu
\right)^{q/q_s^*}
+
\sum_{i\in I_2}
\|\zeta_i^{1/q}u\|_{L^{q_s^*}(M)}^q.
\end{equation}

Fix \(i\in I_1\).
Since the supports of \(\eta_{ij}\) are pairwise disjoint,
\[
\Bigl(\sum_{j=1}^{m_i}\eta_{ij}|u|^q\Bigr)^{q_s^*/q}
=
\sum_{j=1}^{m_i} |\eta_{ij}^{1/q}u|^{q_s^*}.
\]
By the \(G\)-invariance of \(u\), the invariance of \(d\mu\), and the symmetry of the family \(\{\eta_{ij}\}\), we have
\[
\int_M |\eta_{ij}^{1/q}u|^{q_s^*}\,d\mu
=
\int_M |\eta_{i1}^{1/q}u|^{q_s^*}\,d\mu
\qquad \text{for all } j.
\]
Hence
\[
\left(
\int_M
\Bigl(\sum_{j=1}^{m_i}\eta_{ij}|u|^q\Bigr)^{q_s^*/q}\,d\mu
\right)^{q/q_s^*}
=
m_i^{q/q_s^*}\,
\|\eta_{i1}^{1/q}u\|_{L^{q_s^*}(M)}^q.
\]
Substituting this into \eqref{eq5.3}, we obtain
\begin{equation}\label{eq5.4}
\|u\|_{L^{q_s^*}(M)}^q
\le
\sum_{i\in I_1}
m_i^{q/q_s^*}\,
\|\eta_{i1}^{1/q}u\|_{L^{q_s^*}(M)}^q
+
\sum_{i\in I_2}
\|\zeta_i^{1/q}u\|_{L^{q_s^*}(M)}^q.
\end{equation}

We next estimate the terms on the right-hand side.

Fix \(\tau\in(0,1)\).

For \(i\in I_1\), by the sharp Sobolev inequality on \(M\) proved in \cite[Theorem 4.16]{Tan2025SharpFS}, for every \(\rho>0\), there exists \(B_i(\rho)>0\) such that
\[
\|\eta_{i1}^{1/q}u\|_{L^{q_s^*}(M)}^q
\le
\bigl(K(n,s,q)+\rho\bigr)\,[\eta_{i1}^{1/q}u]_{W^{s,q}(M)}^q
+
B_i(\rho)\|\eta_{i1}^{1/q}u\|_{L^q(M)}^q.
\]
Applying Lemma~\ref{lem5.2} with \(\psi=\eta_{i1}\), we obtain
\begin{equation}\label{eq5.5}
\|\eta_{i1}^{1/q}u\|_{L^{q_s^*}(M)}^q
\le
\bigl(K(n,s,q)+\rho\bigr)(1+\tau)^{q-1}
\iint_{M\times M}
\eta_{i1}(x)|u(x)-u(y)|^q K_q^s(x,y)\,d\mu(x)\,d\mu(y)
+
C_i\|u\|_{L^q(M)}^q.
\end{equation}

For \(i\in I_2\), the function \(\zeta_i^{1/q}u\) belongs to \(C_G^\infty(M)\), and
\[
\supp (\zeta_i^{1/q}u)\subset U_i.
\]
By the choice of \(U_i\), Lemma~\ref{lem5.1} applies.
Hence, for every \(\rho_i>0\), there exists \(B_i'>0\) such that
\[
\|\zeta_i^{1/q}u\|_{L^{q_s^*}(M)}^q
\le
\rho_i [\zeta_i^{1/q}u]_{W^{s,q}(M)}^q
+
B_i' \|\zeta_i^{1/q}u\|_{L^q(M)}^q.
\]
Applying Lemma~\ref{lem5.2} with \(\psi=\zeta_i\), we get
\begin{equation}\label{eq5.6}
\|\zeta_i^{1/q}u\|_{L^{q_s^*}(M)}^q
\le
\rho_i(1+\tau)^{q-1}
\iint_{M\times M}
\zeta_i(x)|u(x)-u(y)|^q K_q^s(x,y)\,d\mu(x)\,d\mu(y)
+
C_i'\|u\|_{L^q(M)}^q.
\end{equation}

Substituting \eqref{eq5.5} and \eqref{eq5.6} into \eqref{eq5.4}, we obtain
\begin{align}
\|u\|_{L^{q_s^*}(M)}^q
&\le
\sum_{i\in I_1}
m_i^{q/q_s^*}
\bigl(K(n,s,q)+\rho\bigr)(1+\tau)^{q-1}
\iint_{M\times M}
\eta_{i1}(x)|u(x)-u(y)|^q K_q^s(x,y)\,d\mu(x)\,d\mu(y)
\notag\\
&\quad+
\sum_{i\in I_2}
\rho_i(1+\tau)^{q-1}
\iint_{M\times M}
\zeta_i(x)|u(x)-u(y)|^q K_q^s(x,y)\,d\mu(x)\,d\mu(y)
\notag\\
&\quad+
B\|u\|_{L^q(M)}^q
\label{eq5.7}
\end{align}
for some constant \(B>0\).

For \(i\in I_1\), define
\[
F(x,y)=|u(x)-u(y)|^q K_q^s(x,y).
\]
As before, by the \(G\)-invariance of \(u\), the invariance of \(K_q^s\), and the symmetry of the family \(\{\eta_{ij}\}\),
\[
\iint_{M\times M}\eta_{ij}(x)F(x,y)\,d\mu(x)\,d\mu(y)
=
\iint_{M\times M}\eta_{i1}(x)F(x,y)\,d\mu(x)\,d\mu(y)
\]
for all \(j\).
Therefore
\[
m_i
\iint_{M\times M}\eta_{i1}(x)F(x,y)\,d\mu(x)\,d\mu(y)
=
\sum_{j=1}^{m_i}
\iint_{M\times M}\eta_{ij}(x)F(x,y)\,d\mu(x)\,d\mu(y).
\]
Using this in \eqref{eq5.7}, we get
\begin{align}
\|u\|_{L^{q_s^*}(M)}^q
&\le
\sum_{i\in I_1}
m_i^{q/q_s^*-1}
\bigl(K(n,s,q)+\rho\bigr)(1+\tau)^{q-1}
\sum_{j=1}^{m_i}
\iint_{M\times M}
\eta_{ij}(x)F(x,y)\,d\mu(x)\,d\mu(y)
\notag\\
&\quad+
\sum_{i\in I_2}
\rho_i(1+\tau)^{q-1}
\iint_{M\times M}
\zeta_i(x)|u(x)-u(y)|^q K_q^s(x,y)\,d\mu(x)\,d\mu(y)
\notag\\
&\quad+
B\|u\|_{L^q(M)}^q.
\label{eq5.8}
\end{align}

Now, for \(i\in I_1\),
\[
m_i=\Card O_G^{x_i}\ge k,
\]
and
\[
m_i^{q/q_s^*-1}
=
m_i^{-sq/n}
\le
k^{-sq/n}.
\]
If \(k=+\infty\), then \(I_1=\varnothing\), so the whole first sum is absent, which is consistent with the convention
\[
\frac{K(n,s,q)}{k^{sq/n}}=0.
\]

Since
\[
\sum_{i\in I_1}\sum_{j=1}^{m_i}\eta_{ij}
+
\sum_{i\in I_2}\zeta_i
=1,
\]
it follows from \eqref{eq5.8} that
\begin{align}
\|u\|_{L^{q_s^*}(M)}^q
&\le
k^{-sq/n}\bigl(K(n,s,q)+\rho\bigr)(1+\tau)^{q-1}
\iint_{M\times M}|u(x)-u(y)|^q K_q^s(x,y)\,d\mu(x)\,d\mu(y)
\notag\\
&\quad+
(1+\tau)^{q-1}
\left(\sum_{i\in I_2}\rho_i\right)
\iint_{M\times M}|u(x)-u(y)|^q K_q^s(x,y)\,d\mu(x)\,d\mu(y)
\notag\\
&\quad+
B\|u\|_{L^q(M)}^q.
\label{eq5.9}
\end{align}

Finally, choose \(\tau>0\) and \(\rho>0\) so small that
\[
k^{-sq/n}\bigl(K(n,s,q)+\rho\bigr)(1+\tau)^{q-1}
\le
\frac{K(n,s,q)}{k^{sq/n}}+\frac{\varepsilon}{2}
\]
whenever \(k<+\infty\).
If \(k=+\infty\), this term is absent.
Then choose the numbers \(\rho_i>0\), \(i\in I_2\), so that
\[
(1+\tau)^{q-1}\sum_{i\in I_2}\rho_i\le \frac{\varepsilon}{2}.
\]
It now follows from \eqref{eq5.9} that
\[
\|u\|_{L^{q_s^*}(M)}^q
\le
\left(\frac{K(n,s,q)}{k^{sq/n}}+\varepsilon\right)
[u]_{W^{s,q}(M)}^q
+
B\|u\|_{L^q(M)}^q
\qquad \text{for all } u\in C_G^\infty(M).
\]
By the density of \(C_G^\infty(M)\) in \(W_G^{s,q}(M)\), the inequality extends to every \(u\in W_G^{s,q}(M)\).
This proves Theorem~\ref{thm1.4}.
\end{proof}

\section*{Acknowledgments}
%We would like to thank the anonymous referee for his/her careful readings of our manuscript and the useful comments. 

\medskip
{\bf Funding:} This work is supported by National Natural Science Foundation of China (12301145, 12261107, 12561020) and Yunnan Fundamental Research Projects (202301AU070144, 202401AU070123).

\medskip
{\bf Author Contributions:} All the authors wrote the main manuscript text together and these authors contributed equally to this work.

\medskip
{\bf Data availability:}  Data sharing is not applicable to this article as no new data were created or analyzed in this study.

\medskip
{\bf Conflict of Interests:} The authors declares that there is no conflict of interest.

\bibliographystyle{plain} % 设置参考文献格式
\bibliography{ref} % 指定参考文献文件
\end{document}